\documentclass{amsart}
\usepackage[]{latexsym,amssymb,amsmath,amsfonts, amsthm}
\usepackage[all,cmtip,ps]{xy}

\def\C{\mathcal{C}}

\def\X{\mathcal{X}}
\def\Y{\mathcal{Y}}

\def\H{\mathcal{H}}

\def\N{{\mathbb N}}

\def\ld{\varinjlim}
\def\End{\operatorname{End}}

\def\Hom{\operatorname{Hom}}
\def\Ext{\operatorname{Ext}}
\def\pdim{\operatorname{proj\,dim}}
\def\t{\operatorname{t}}
\def\dualita#1#2{\mathrel{
                 \mathop{\vcenter{
                 \offinterlineskip
                 \hbox to 0.6truecm{\rightarrowfill}
                 \hbox to 0.6truecm{\leftarrowfill}}}%
                 \limits_{#2}^{#1}}}

\DeclareMathOperator{\Prod}{Prod}

\DeclareMathOperator{\Ker}{Ker}

\newtheorem{theorem}{Theorem}[section]

\newtheorem{corollary}[theorem]{Corollary}

\newtheorem{definition}[theorem]{Definition}

\newtheorem{lemma}[theorem]{Lemma}

\newtheorem{RRcond}[theorem]{Reiten-Ringel Condition}
\newtheorem{problem}[theorem]{Problem}
\newtheorem{proposition}[theorem]{Proposition}
\theoremstyle{remark}
\newtheorem{remark}[theorem]{Remark}

\newcommand*{\rMod}{\textrm{\textup{Mod-}}}
\newcommand*{\rmod}{\textrm{\textup{mod-}}}
\newcommand{\FP}[1]{\textrm{\textup{FP}}_{#1}}
\begin{document}

\title{Cotorsion pairs, torsion pairs, and $\Sigma$-pure injective cotilting modules}
\author{Riccardo Colpi, Francesca Mantese, Alberto Tonolo}
\address[R. Colpi]{ Dip. Matematica Pura ed Applicata, Universit\`a degli studi di Padova, via Trieste 63, I-35121 Padova Italy}
\email{colpi@math.unipd.it}
\address[F. Mantese]{Dipartimento di Informatica, Universit\`a degli Studi di Verona, strada Le Grazie  15, I-37134 Verona - Italy}
\email{francesca.mantese@univr.it}
\address[A. Tonolo]{ Dip. Matematica Pura ed Applicata, Universit\`a degli studi di Padova, via Trieste 63, I-35121 Padova Italy}
\email{tonolo@math.unipd.it}
\thanks{Research supported by grant CPDA071244/07 of Padova University}
\dedicatory{to the memory of our friend and colleague Silvia Lucido}

\maketitle

\begin{abstract}

In this paper we study cotorsion and torsion pairs induced by cotilting modules. We prove the existence of a strong relationship between the $\Sigma$-pure injectivity of the cotilting module and the property of the induced cotorsion pair to be of finite type. In particular for cotilting modules of injective dimension at most 1, or for noetherian rings, the two notions are equivalent.
On the other hand we prove that a torsion pair  is cogenerated by a $\Sigma$-pure injective cotilting module if and only if its heart is a locally noetherian Grothendieck category.
Moreover we prove that any ring admitting a $\Sigma$-pure injective cotilting module of injective dimension at most 1 is necessarily coherent. Finally, for noetherian rings, we characterize cotilting torsion pairs induced by $\Sigma$-pure injective cotilting modules.
%

%
\end{abstract}

\section*{Introduction}
The class of modules over an arbitrary associative ring $R$ is too complex to admit any satisfactory classification. For this reason, usually one restricts to study particular, possibly large and representative, classes of modules.

In the recent literature, the theory of modules widely uses the notions of torsion and cotorsion pairs.
Torsion and cotorsion pairs are couples $(\mathcal L,\mathcal M)$ of classes of modules which are maximal with respect to the orthogonality conditions $\Hom(\mathcal L,\mathcal M)=0$ and $\Ext(\mathcal L,\mathcal M)=0$, respectively. These pairs are partially ordered by inclusion of their first components, forming complete lattices. The study of their properties allows an approximation of the whole category of modules.

In this paper we concentrate on torsion and cotorsion pairs \emph{induced} by a cotilting $R$-module, in the sequel briefly called cotilting torsion  and cotorsion pairs. In particular we study finiteness properties of cotilting torsion and cotorsion pairs. Any cotilting module is pure injective; following a suggestion of Enrico Gregorio, we will focus on torsion and cotorsion pairs induced by $\Sigma$-pure injective cotilting modules.

In the first section we compare an arbitrary cotorsion pair $(\mathcal L,\mathcal M)$ with the cotorsion pairs generated by the $\ell$-presented modules in $\mathcal L$.

In the second section we analyze cotilting cotorsion pairs,  obtaining new characterizations of those of \emph{finite type} (see Theorem~\ref{prop:Y1}). In particular, in the noetherian case, these are exactly those induced by a $\Sigma$-pure injective cotilting module (see Corollary~\ref{cor:2.5}). 

Given a torsion pair $(\mathcal X,\mathcal Y)$ in the category of right $R$-modules, 
the \emph{heart} of the torsion pair $(\mathcal X,\mathcal Y)$ is an abelian subcategory $\mathcal H(\mathcal X,\mathcal Y)$ of the derived category of right $R$-modules (see section~\ref{cuore} for more details).
Recently, in \cite{CG} it has been proved that  $\mathcal H(\mathcal X,\mathcal Y)$ is a Grothendieck category if and only if $(\mathcal X,\mathcal Y)$ is a cotilting torsion pair. In the third section we prove (see Theorem~\ref{thm:heart}) that $\mathcal H(\mathcal X,\mathcal Y)$ is a \emph{locally noetherian} Grothendieck category if and only if $(\mathcal X,\mathcal Y)$ is cogenerated by a $\Sigma$-pure injective cotilting module. Moreover, we get that any ring $R$ admitting a $\Sigma$-pure injective cotilting module of injective dimension at most 1 is necessarily coherent (see Corollary~\ref{coro:coherent}).

Finally, in the fourth section, we study cotilting torsion pairs over a noetherian ring, giving a complete characterization of  those induced by a $\Sigma$-pure injective cotilting module. In particular we prove (see Theorem~\ref{thm:torsionpairs}) that these are exactly the cotilting torsion pairs which satisfy the \emph{Reiten-Ringel condition} (see Condition~\ref{RRcond}).
This was originally introduced for finite dimensional $k$-algebras in \cite{RR} as a sufficient condition to guarantee that the closure under direct limits of a splitting torsion pair in the category of finitely generated modules is a splitting torsion pair in the category of all modules.

\section*{Notation and terminology}

Let $R$ be a ring. We denote by $\rMod R$ the category of right $R$-modules and by $\FP \ell$ the subcategory of the right $\ell$-presented $R$-modules, i.e. the modules $M$ in $\rMod R$ which admit a resolution
\[P_\ell\to ...\to P_1\to P_0\to M\to 0\]
where $P_i$ is a finitely generated projective module for $0\leq i\leq \ell$. Denote by $\rmod R$ the intersection $\cap_{\ell\in\mathbb N}\FP \ell$. In particular $\FP 0$ and $\FP 1$ are the categories of finitely generated and of finitely presented right $R$-modules, respectively.

For any module $C\in \rMod R$,  $\Prod C$  denotes the class of all direct summands of direct products of copies of $C$.

Given a class $\mathcal C\subseteq \rMod R$, we define the following:
 $${\mathcal C}^{\perp}=\{M\in \rMod R \mid \Ext_R^1(C, M)=0 \ \text{for any} \ C\in \mathcal C\}$$
 $${\mathcal C}^{\perp_{\infty}}=\{M\in \rMod R \mid \Ext_R^n(C, M)=0 \ \text{for any} \ C\in \mathcal C \ \text{and for any} \ n> 0\}. $$
  Similarly we define $^{\perp}{\mathcal C}$ and $^{\perp_{\infty}}{\mathcal C}$ .  

A module $U\in \rMod R$ is an \emph{$n$-cotilting module}  if  $\Ext^i_R(U^\alpha, U)=0$ for any $i>0$ and all cardinals $\alpha$, $U$ has injective dimension at most $n$, and there exists a long exact sequence $0\to U_n \to \dots \to U_0\to W\to 0$ where $W$ is an injective cogenerator of $\rMod R$ and $U_i\in \Prod U$ for $i=0,\dots, n$.

Any cotilting module $U$ is pure-injective \cite{B, S}, so the class $\mathcal Y={^{\perp_{\infty}}U}$ is closed under direct limits. Moreover, if  $U$ has  injective dimension at most one, then $\mathcal Y$ coincides with the class of modules cogenerated by $U$; it is a torsion-free class and the corresponding torsion pair $(\mathcal X, \mathcal Y)$ is called the \emph{cotilting torsion pair} \emph{cogenerated by} $U$. Notice that the latter is a \emph{faithful} torsion pair, i.e. $R_R$ belongs to $\mathcal Y$.

Let $\mathcal A, \mathcal B\subseteq \rMod R$. The pair $(\mathcal A, \mathcal B)$ is called a \emph{cotorsion pair} if $\mathcal A={^{\perp}{\mathcal B}}$ and $\mathcal B={\mathcal A}{^{\perp}}$.  If $\mathcal B={\mathcal C}^{\perp}$ for a class of modules $\mathcal C$, we say that the cotorsion pair is \emph{generated} by $\mathcal C$. Similarly, if  $\mathcal A={^{\perp}{\mathcal C}}$ we say that the cotorsion pair is \emph{cogenerated} by $\mathcal C$. Moreover $(\mathcal A, \mathcal B)$ is said of \emph{finite type} if it is generated by a set of modules in $\rmod R$.
A cotorsion pair is called \emph{hereditary} if $\mathcal A={^{\perp_{\infty}}{\mathcal B}}$ and $\mathcal B={\mathcal A}{^{\perp_{\infty}}}$. A cotorsion pair  $(\mathcal A, \mathcal B)$ is hereditary if and only if $\mathcal A$ is a \emph{resolving class}, i.e. it is closed under extensions, kernels of epimorphisms and it contains the projectives,  or equivalently, if $\mathcal B$ is a \emph{coresolving class}, i.e., it is closed under extensions, cokernels of monomorphisms and  it contains the injectives \cite{GR}. 
A cotorsion pair is called \emph{complete} if $\mathcal A$ provides special precovers or, equivalently, if $\mathcal B$ provides special preenvelopes (see \cite[Ch. 2]{GT}).

 If $U$ is a cotilting module and $\mathcal Y={^{\perp_{\infty}}U}$, then the cotorsion pair $(\mathcal Y, \mathcal Y^{\perp})$ generated by $\mathcal Y$ is hereditary and complete \cite[Ch. 8]{GT}; its kernel $\mathcal Y\cap \mathcal Y^{\perp}$ coincides with $\Prod U$ \cite[Lemma~8.1.4]{GT}. In the sequel we will refer to $(\mathcal Y, \mathcal Y^{\perp})$ as the \emph{cotilting cotorsion pair  induced  by }$U$. 
 Finally, two cotilting modules are called \emph{equivalent} if they induce the same cotorsion pair.

Let $\sigma$ be an ordinal. An increasing chain of submodules $( M_{\alpha} \mid \alpha\leq \sigma)$ of a module $M$ is called a \emph{filtration} of $M$ provided that $M_0=0$, $M_{\alpha}=\cup_{\beta<\alpha} M_{\beta}$ for any limit ordinal $\alpha\leq \sigma$, and $M_{\sigma}=M$. Given a class of modules $\mathcal C$, a \emph{$\C$-filtration} of $M$ is a filtration such that, for any $\alpha<\sigma$, $M_{\alpha+1}/M_{\alpha}$ is isomorphic to some element of $\mathcal C$.

Filtrations play an important role in the study of cotorsion pairs, as widely described in \cite{GT}.  In particular, in the sequel we will often refer to the following version of Hill Lemma, stated and proved in a more general form  in \cite[Theorem~4.2.6]{GT}.

\begin{theorem}[Hill Lemma]\label{lemma:hill}
Let $R$ be  a ring and $\mathcal C$ a set of finitely presented modules. Let $M$ be a module with a $\mathcal C$-filtration $\mathcal M=( M_{\alpha} \mid \alpha\leq \sigma )$. 
Then there exists a family $\mathcal F$ of submodules of $M$ such that:
\begin{enumerate}
\item $\mathcal M\subseteq \mathcal F$;
\item $\mathcal F$ is closed under arbitrary sums and intersections;
\item For any $F_1$ and $F_2$ in $\mathcal F$ such that $F_1\leq F_2$, the module $F_2/F_1$ admits a $\mathcal C$-filtration.
\item For any finitely generated submodule $L$ of $M$, there exists $F\in \mathcal F$ such that $L\leq F$ and $F$ admits a finite $\C$-filtration. In particular, $F$  is finitely presented. 
\end{enumerate}
\end{theorem}

\section{Comparing cotorsion pairs}

Given a cotorsion pair $\mathfrak A:=(\mathcal A, \mathcal A^{\perp})$,
we denote by $\mathcal A_i$ the class of modules in $\mathcal A$ which belong to $\FP i$, $i\geq 0$, and by $\mathcal A_{\infty}$ the class of modules in $\mathcal A$ which belong to $\rmod R$.
Beside $\mathfrak A=(\mathcal A, \mathcal A^{\perp})$, we consider the cotorsion pairs $\mathfrak A_i=({}^{\perp}(\mathcal A_i^{\perp}), \mathcal A_i^{\perp})$ generated by the set $\mathcal A_i$, $i\geq 0$, and $\mathfrak A_{\infty}=({}^{\perp}(\mathcal A_{\infty}^{\perp}), \mathcal A_{\infty}^{\perp})$ generated by the set $\mathcal A_{\infty}$. Considering the partial order induced by the inclusion of their first components, we have the following chain of cotorsion pairs:
\[\mathfrak A_{\infty}\leq...\leq\mathfrak A_{i+1}\leq\mathfrak A_{i}\leq ... \mathfrak A_1\leq \mathfrak A_0\leq \mathfrak A.
\]
The cotorsion pairs $\mathfrak A_i$, $0\leq i\leq \infty$, are complete, and their cotorsion classes ${}^{\perp}(\mathcal A_i^{\perp})$ consist of all direct summands of $\mathcal A_i$-filtered modules (see \cite[Theorem~3.2.1 and Corollary~3.2.4]{GT}). Moreover, by definition, the cotorsion pair $\mathfrak A_{\infty}$ is of {finite type}.

\begin{proposition}\label{prop:generale}
Let $\mathfrak A=(\mathcal A,\mathcal A^\perp)$ be a cotorsion pair, and $\mathcal S$ be any set of modules in $\mathcal A$. Assume $\mathcal A$ is a resolving class.
Then the following are equivalent:
\begin{enumerate}
\item $\mathcal S^{\perp}=
\mathcal A^\perp$;
\item $\mathcal S^{\perp}\cap\mathcal A=
\mathcal A^\perp\cap\mathcal A$.
\end{enumerate}
%
\end{proposition}
\begin{proof}
Clearly 1 implies 2. Assume 2 holds; $\mathcal S^{\perp}\supseteq
\mathcal A^\perp$ is always true. Therefore
$\mathfrak S=({}^\perp(\mathcal S^\perp), \mathcal S^\perp)\leq (\mathcal A,\mathcal A^\perp)$.
 Let $M$ be in $\mathcal A$. By \cite[Theorem 3.2.1]{GT} $\mathfrak S$ is a complete cotorsion pair, therefore the class ${}^\perp(\mathcal S^\perp)$ gives special precovers. Let
\[0\to S_1\to S_2\to M\to 0\]
be a ${}^\perp(\mathcal S^\perp)$-special precover of $M$. Then $S_2$ belongs to ${}^\perp(\mathcal S^\perp)\subseteq\mathcal A$ and $S_1$ belongs to  $\mathcal S^\perp$. Since $\mathcal A$ is resolving, the module $S_1$ belongs to \[\mathcal S^\perp\cap\mathcal A=\mathcal A^\perp\cap\mathcal A.\]
Then the above exact sequence splits and $M$ is a direct summand of $S_2$, so $M$ belongs to ${}^\perp(\mathcal S^\perp)$. Therefore $\mathfrak S=(\mathcal A,\mathcal A^\perp)$ and we conclude $\mathcal S^{\perp}=
\mathcal A^\perp$.
%

\end{proof}

\begin{corollary}
Let $\mathcal A$ be a resolving class and $0\leq i\leq \infty$.
Then $\mathfrak A_i=\mathfrak A$ if and only if
\[\mathcal A_i^{\perp}\cap \mathcal A=\mathcal A^{\perp}\cap \mathcal A.\]
\end{corollary}

\begin{proposition}\label{prop:A_{i+n}=A_i}
Let $i\geq 0$. Then $\mathfrak A_i=\mathfrak A_{i+1}$ if and only if $\mathcal A_i=\mathcal A_{i+1}$. If moreover $\mathcal A$ is a resolving class, then $\mathfrak A_{i}=\mathfrak A_{\infty}$.
\end{proposition}
\begin{proof}
By \cite[Corollary~3.2.4]{GT}, ${}^\perp(\mathcal A_{i+1}^\perp)={}^\perp(\mathcal A_{i}^\perp)$ consists of all direct summands of $\mathcal A_{i+1}$-filtered modules. In particular any module $A$ in $\mathcal A_i$ is a direct summand of a $\mathcal A_{i+1}$-filtered module. Since $A$ is finitely generated, by  Theorem~\ref{lemma:hill}  it is a direct summand of a finitely presented module $A'$ admitting a finite $\mathcal A_{i+1}$-filtration. Therefore, since $A'$ is a finite extension of modules in $\mathcal A_{i+1}$ and $A$ is a direct summand of $A'$, both $A$ and $A'$ belong to $\mathcal A_{i+1}$. Then $\mathcal A_i$ is contained in $\mathcal A_{i+1}$ and hence $\mathcal A_{i+1}=\mathcal A_i$.

Finally, let us assume $\mathcal A$ resolving and  let $M$ belong to $\mathcal A_{i+1}$; then there exists an exact sequence
\[P_{i+1}\to P_{i}\to ...\to P_0\to M\to 0\]
where $P_\ell$ is a finitely generated projective module for $0\leq \ell\leq i+1$.
Let us denote by $I_\ell$ the image of $P_\ell\to P_{\ell-1}$; since $\mathcal A$ is closed under kernels of epimorphisms, it is easy to prove recursively that $I_\ell$ belongs to $\mathcal A_{i+1-\ell}$. In particular $I_1$ belongs to $\mathcal A_{i}=\mathcal A_{i+1}$, and hence the finitely generated projective resolution of $M$ can be continued with one more step on the left. Repeating this argument, we conclude that $M$ belongs to $\mathcal A_{\infty}$.
\end{proof}
%


\begin{proposition}\label{prop:limiti}
Let $i\geq 1$; if $\mathfrak A=\mathfrak A_i$, then $\mathcal A\subseteq \{\text{direct summands of }\varinjlim \mathcal A_i\}$.
\end{proposition}
\begin{proof}
If $\mathfrak A=\mathfrak A_i$, then $\mathcal A={}^\perp(\mathcal A_i^\perp)$. Therefore any object $A$ in $\mathcal A$ is a direct summand of a $\mathcal A_i$-filtered module $\overline A$. The module $\overline A$ is the direct limit of its finitely generated submodules: $\overline A=\varinjlim_{\ell\in\Lambda} F_\ell$. By Theorem~\ref{lemma:hill}, for each $\ell \in\Lambda$ there exists a submodule $G_\ell$ of $\overline A$ in $FP_i$, and containing $F_\ell$, such that $G_\ell$ is $\mathcal A_i$-filtered. This gives $\overline A=\varinjlim_{\ell\in\Lambda} F_\ell=\varinjlim_{\ell\in\Lambda} G_\ell$, and hence the thesis.
\end{proof}

\section{Cotilting cotorsion pairs}

In all this section we assume that $\mathfrak Y=(\mathcal Y,\mathcal Y^\perp)$ is a cotorsion pair induced by an $n$-cotilting module $U$. 


\begin{lemma}\label{lemma:limiti}
Let $0\leq i\leq \infty$. If $U_R$ is $\Sigma$-pure-injective and $\mathcal Y=\varinjlim \mathcal Y_i$, then $\mathfrak Y_i=\mathfrak Y$. In particular if $U_R$ is a $\Sigma$-pure-injective 1-cotilting module, then $\mathfrak Y_0=\mathfrak Y$.
\end{lemma}
\begin{proof}
Let $M$ be in $\mathcal Y_i^{\perp}\cap\mathcal Y$; there exists a short exact sequence
\[0\to M\to U^{\beta}\to U'\to 0\]
with $U'$ in $\mathcal Y$ (see \cite[Proposition~8.1.5]{GT}). By assumption, $U'=\varinjlim_{
\lambda\in \Lambda}U'_{\lambda}$ with the $U'_{\lambda}$ in $\mathcal Y_i$. Consider the pullback diagram
\[
\xymatrix{
0\ar[r]&M\ar[r]&U^\beta\ar[r]&U'\ar[r]&0
\\
0\ar[r]&M\ar@{=}[u]\ar[r]&P_\lambda\ar[r]\ar[u]&U'_\lambda\ar[r]\ar[u]&0
}
\]
Since $M\in \mathcal Y_i^{\perp}$ the lower exact sequence splits. Therefore the upper exact sequence is a direct limit of splitting exact sequences, and hence it is pure. Since $U_R$ is $\Sigma$-pure-injective, by \cite[Corollary~8.2]{JL}, also the upper short exact sequence splits; then $M$ belongs to $\Prod U=\mathcal Y^\perp\cap\mathcal Y$. By Proposition~\ref{prop:generale}, we conclude $\mathcal Y_i^{\perp}=\mathcal Y^\perp$, and hence the thesis.
\end{proof}

\begin{theorem}\label{prop:Y1}
The following are equivalent for a cotorsion pair $\mathfrak Y$ induced by an $n$-cotilting module $U$:
\begin{enumerate}
\item $\mathfrak Y_1=\mathfrak Y$;
\item $\mathfrak Y_{\infty}=\mathfrak Y$, i.e. $\mathfrak Y$ is of finite type;
\item $\mathfrak Y_i=\mathfrak Y$ for some $i\geq 1$;
\item $U$ is $\Sigma$-pure-injective and $\mathcal Y=\varinjlim \mathcal Y_1$;
\end{enumerate}
In particular in such a case $FP_n=FP_{n+1}=\rmod R$.
\end{theorem}
\begin{proof}
$1\Rightarrow 2$:  since $\mathcal Y$ is a resolving class, by Proposition~\ref{prop:A_{i+n}=A_i} we have $\mathfrak Y_1=\mathfrak Y_\infty$.

$2\Rightarrow 3$: is clear.

$3\Rightarrow 4$: Since $U$ is pure injective, every module in $\Prod U$ is pure injective. Therefore, $\mathcal Y_1^{\perp}\cap\mathcal Y=\Prod U$ being closed under arbitrary direct sums, $U^{(\alpha)}$ is pure injective for each cardinal $\alpha$. Thus $U$ is $\Sigma$-pure-injective. Moreover, since $\mathcal Y$ is closed under direct limits and direct summands, by Proposition~\ref{prop:limiti} we have $\mathcal Y=\varinjlim \Y_1$.

$4\Rightarrow 1$: it follows immediately by Lemma~\ref{lemma:limiti}.


Finally, let $M$ be a module in $FP_{n}$. Consider the finitely generated projective resolution
\[P_n\stackrel{f_n}{\to} P_{n-1}\to\dots\to P_0\stackrel{f_0}{\to} M\to 0\]
Since the injective dimension of $U$ is $\leq n$, by dimension shifting for any $i\geq 1$
\[\Ext^i_R(\Ker f_{n-1},U)\cong \Ext^{i+n-1}_R(\Ker f_{0},U)\cong \Ext^{i+n}_R(M,U)=0.\]
Therefore the finitely generated module $\Ker f_{n-1}$ belongs to ${}^{\perp_\infty} U=\mathcal Y$. Since $\mathcal Y_0=\mathcal Y_1$, the module $\Ker f_{n-1}$ is finitely presented; thus $\Ker f_n$ is finitely generated and $M$ belongs to $FP_{n+1}$.
\end{proof}



\begin{problem}\label{problema}
Are there $\Sigma$-pure-injective $n$-cotilting modules inducing a cotorsion pair $\mathfrak Y$ which is not of finite type, i.e. $\mathfrak Y_1\not=\mathfrak Y$?
\end{problem}
In \cite[Corollary~4.11]{AST} there is a negative answer in the case
$R$ is right noetherian. The problem has a negative answer also in the general case if $n=1$:  

\begin{corollary}\label{cor:2.5}
If $n=1$, then $\mathfrak Y$ is of finite type if and only if the $1$-cotilting module $U$ is $\Sigma$-pure injective.
\end{corollary}
\begin{proof}
The proof is an immediate consequence of the forthcoming Corollary~\ref{coro:coherent}.
\end{proof}

\section{Cotilting torsion pairs}\label{cuore}

Beyond inducing the cotorsion pair $(\mathcal Y, \mathcal Y^\perp)$, a 1-cotilting right $R$-module cogenerates a torsion pair $(\mathcal X,\mathcal Y)$.  Among torsion pairs, the cotilting ones are precisely those for which the torsion-free class gives special precovers \cite[Theorem~2.5]{ATT}.  Note that two equivalent 1-cotilting modules  cogenerate  the same torsion pair. 

Studying splitting torsion pairs over finite dimensional $k$-algebras, 
Reiten and Ringel in \cite{RR} introduced the following finiteness condition on a torsion pair $(\mathcal X,\mathcal Y)$:
\begin{RRcond}\label{RRcond}
If $Y\in\Y$ has a finitely generated submodule $0\neq Y_0\leq Y$ such that $Y/Y_0\in\X$, then $Y$ is finitely generated.
\end{RRcond}
This condition turns out to have interesting applications also in a more general setting, as we will show in the next results. 

\begin{proposition}\label{prop:Ringelcondition}
If a cotilting torsion pair $(\mathcal X,\mathcal Y)$ satisfies the Reiten-Ringel condition, then the cotorsion pairs $\mathfrak Y$ and $\mathfrak Y_0$ coincide.
\end{proposition}
\begin{proof} Let us prove that $\mathcal Y^{\perp}=\mathcal Y_0^\perp$. Let $M$ belongs to $\mathcal Y_0^\perp$; we have to show that any short exact sequence
$$
0 \to M \to E \to Y \to 0
$$
with $Y\in \mathcal Y$, splits.
Set $\mathcal D = \{D\leq E \mid M\cap D=0,\ E/(M\oplus D)\in\Y\}$.
Since $\Y$ is closed under direct limits, any ascending
chain in $\mathcal D$ has union in $\mathcal D$, so that $\mathcal D$ contains a maximal element. Let's call it $D_{max}$:
 the goal consists in proving that $M\oplus D_{max} = E$. Suppose that this is not the case, and let
$M\oplus D_{max}<E'\leq E$, with $E'/(M\oplus D_{max})$ finitely generated.
Let $E''/E' = \t_{\X}(E/E')\in\X$, where $\t_{\X}$ denotes the torsion radical associated to $\X$.
From the exact sequence
$$
0 \to E'/(M\oplus D_{max}) \to E''/(M\oplus D_{max}) \to E''/E' \to 0
$$
since $E''/(M\oplus D_{max})\leq E/(M\oplus D_{max})\in\Y$,
the module $E''/(M\oplus D_{max})$ is finitely generated  by the Reiten-Ringel condition.

By assumption, the exact sequence
$$
0 \to M\cong (M\oplus D_{max})/D_{max} \to E''/D_{max} \to E''/(M\oplus D_{max}) \to 0
$$
splits, with $E''/(M\oplus D_{max})  \neq 0$. Thus there exists a module $D'$, with $D_{max}<D'\leq E''$, such that
$E''/D_{max} = ((M\oplus D_{max})/D_{max}) \oplus (D'/D_{max})$.

In particular $M\cap D' \leq M\cap D_{max} = 0$ and $M\oplus D'=E''$.
Finally, $E/(M\oplus D') = E/E'' \cong (E/E')/\t_{\X}(E/E') \in \Y$, contrary to the maximality of $D_{max}$.
\end{proof}

The \emph{heart} $\mathcal H(\mathcal X, \mathcal Y)$ of a torsion pair $(\mathcal X, \mathcal Y)$ is the abelian subcategory of the derived category of $\rMod R$ whose objects are the complexes which have zero cohomologies everywhere, except for degrees $0$ and $-1$ where they have cohomologies in $\mathcal X$ and in $\mathcal Y$, respectively.
In \cite{CF} it is proved that, if $(\mathcal X,\mathcal Y)$ is faithful, the stalk complex $V:=R[1]$ is a tilting object in $\mathcal H(\mathcal X, \mathcal Y)$ with endomorphisms ring $R$; moreover it determines a torsion pair $(\mathcal T, \mathcal F)$ and  a pair of equivalences $H_V:\mathcal T\dualita{}{}\mathcal Y:T_V$ and 
$H'_V:\mathcal F\dualita{}{}\mathcal X:T'_V$ where $H_V=\Hom(V, -)$, $H'_V=\Ext(V, -)$ and $T_V$, $T'_V$ are their adjoint functors.

In  \cite{CG}  and \cite{CGM}  it has been proved that a faithful torsion pair is cotilting if and only if the associated  heart is  a Grothendieck category.  
In the next theorem we show that a 1-cotilting module is $\Sigma$-pure injective if and only if the corresponding faithful torsion pair has a locally noetherian heart (see~\cite[Section~V.4]{St}).

\begin{theorem}\label{thm:heart}
A faithful torsion pair $(\X, \Y)$ in $\rMod R$ is cogenerated by a
1-cotilting  $\Sigma$-pure injective right $R$-module if and only if the  heart $\mathcal H(\mathcal X, \mathcal Y)$ is a locally noetherian Grothendieck category. 
%
\end{theorem}
\begin{proof}
By \cite{CG}  and \cite{CGM},  $\mathcal H:=\mathcal H(\mathcal X, \mathcal Y)$ is a Grothendieck category if and only if $(\X, \Y)$ is cogenerated by a 1-cotilting module $C_R$.

Let us assume that  $\mathcal H$ is  locally noetherian.
First let us show that the torsion pair $(\X,\Y)$ satisfies the Reiten-Ringel condition. Indeed, 
let $0 \to Y_0 \to Y \to X \to 0$ be an exact sequence with $Y\in\Y$, $X\in\X$ and $Y_0$ finitely generated.
We get the exact sequence
$0 \to T'_{V}X \to T_{V}Y_0 \to T_{V}Y \to 0$. Since $Y_{0}$ is finitely generated over $R$,
we see that $T_{V}Y_{0}$  is a factor of $V^n$, for some $n\in\N$. Following \cite[Corollary~4.3]{C},
we have that $V$ is finitely presented, so that $T_{V}Y_{0}$
is finitely generated. Thus $T_{V}Y$ is finitely generated. Since the functor $H_{V}$ carries finitely generated objects of $\mathcal H$ to finitely generated $R$-modules \cite[Lemma~6.1]{C}, the module $Y \cong H_{V}T_{V}Y$
is finitely generated. So the Reiten-Ringel condition  is satisfied and, by Proposition~\ref{prop:Ringelcondition} we get that $\mathfrak Y=\mathfrak Y_0$. Finally, let us prove that $\mathcal Y_0=\mathcal Y_1$; then $\mathfrak Y=\mathfrak Y_1$ and we conclude by Theorem~\ref{prop:Y1}. Indeed, let $F\in\mathcal Y_0$ and $0\to K\to R^n\to F\to 0$  be an exact sequence in $\mathcal Y$.  Then we obtain the exact sequence $0\to T_{V} K\to V^n\to T_V F\to 0$ and, since $\mathcal H$ is locally noetherian, $T_V K$ is finitely generated. Thus $K\cong H_{V} T_{V} K$ is finitely generated and so $F$ is finitely presented. 

The proof of the converse implication follows the same arguments used in  \cite{CG}. 
First note that if $C$ is  $\Sigma$-pure injective, then $\Prod C$ is closed under direct sums. Moreover,  in the Grothendieck category $\mathcal H$,  an object $I$   is injective if and only if $I=T_V(C')$, for some $C'\in\Prod C$  \cite[Proposition~3.8]{CGM}. From these easy observations, it follows that in $\H$, the class of injective objects is closed under coproducts. Indeed let $I_{\lambda}$,  $\lambda\in \Lambda$,  a set of injective objects: then  $\oplus I_{\lambda}=\oplus T_V(C'_{\lambda})=T_V(\oplus C'_{\lambda})$, since $T_V$ commutes with coproducts.  Hence, following the proof of \cite[Proposition~V.4.3]{St},  one gets that any small object in $\H$ is noetherian. Finally, in \cite[Lemma~3.4]{CGM} it is is shown that the set $\{ Z\leq V^n, \  n\in \mathbb N \}$ generates $\H$: since $V$ is small and so noetherian, $\H$ admits a set of noetherian generators. 
%
\end{proof}

The following corollary shows that the assumption that a ring admits a $\Sigma$-pure injective 1-cotilting module is very strong, since it  implies the ring is coherent.

\begin{corollary}\label{coro:coherent}
Let $C$ be a 1-cotilting right $R$-module and $\mathfrak Y$ the corresponding cotorsion pair. If $C$ is $\Sigma$-pure-injective, then:
\begin{enumerate}
\item $\mathfrak Y=\mathfrak Y_1$;
\item  the ring $R$ is right coherent.
\end{enumerate}
\end{corollary}
\begin{proof}
 Following the proof of Theorem~\ref{thm:heart},  we get that $\mathfrak Y=\mathfrak Y_0=\mathfrak Y_1$ and so,  by
Theorem~\ref{prop:Y1}, the ring $R$ is right coherent.
\end{proof}

\begin{problem}
Is a ring admitting a $\Sigma$-pure-injective $n$-cotilting module necessarily coherent?
\end{problem}


It follows also a nice application to the tilting setting: any classical 1-tilting module (for the definition see~\cite[Ch. 5]{GT}) over a right noetherian ring has right coherent endomorphism ring.

\begin{corollary}
Let $T_S$ be a classical $1$-tilting module. The following are equivalent:
\begin{enumerate}
\item  $S$ is noetherian;  
\item  $\Hom_S(T, W)$ is a $\Sigma$-pure injective 1-cotilting $\End_S(T)$-module for any injective cogenerator $W$ of $\rMod S$.
\end{enumerate}
In such a case, $\End_S(T)$ is a right coherent ring.
\end{corollary}
\begin{proof}
It is known \cite[Theorem~2.3]{CGM} that the category  $\rMod S$ is equivalent to the heart of the cotilting  torsion pair $(\X,\Y)$ cogenerated by  $\Hom_S(T, W)$. Then the result follows from   Theorem~\ref{thm:heart}.
\end{proof}

\section{Cotilting torsion pairs for noetherian rings}


As we have seen in the previous sections, the notion of $\Sigma$-pure injectivity for a cotilting module is closely related to finiteness conditions  on the rings and on the classes involved.  The aim of this section is to characterize the cotilting torsion pairs cogenerated by a $\Sigma$-pure injective 1-cotilting module in the noetherian setting, and investigate their finiteness properties. 

For the rest of this section, $R$ denotes a right noetherian ring. Buan and Krause in \cite{BK} proved that cotilting torsion pairs play a relevant role passing from the category of finitely generated $R$-modules to the whole category of $R$-modules: indeed there is a bijective correspondence between cotilting torsion pairs $(\mathcal X, \mathcal Y)$ in $\rMod R$ and faithful torsion pairs $(\mathcal X_0, \mathcal Y_0)$ in $\rmod R$. The correspondence is given by the mutually inverse assignments
\[
\mathcal X \mapsto \mathcal X_0 = \mathcal X\cap \rmod R,\qquad
\mathcal X_0 \mapsto  \mathcal X=\varinjlim \mathcal X_0\]
and
\[
\mathcal Y \mapsto \mathcal Y_0=\mathcal Y\cap \rmod R,\qquad
\mathcal Y_0 \mapsto \mathcal Y=\varinjlim \mathcal Y_0.\]
In case the torsion pair $(\mathcal X, \mathcal Y)$ is cogenerated by a $\Sigma$-pure injective 1-cotilting module, this correspondence preserves the following relevant (see the notions of quasitilted artin algebras \cite{HRS} and of quasitiled rings \cite{CF}) properties:
\begin{proposition}\label{GT-equiv_enough}
Let $C_R$ be a $\Sigma$-pure injective 1-cotilting module and let $\mathcal Y={}^{\perp}C$. Then:
\begin{enumerate}
\item $(\X_0, \Y_0)$ splits if and only if $(\X, \Y)$ splits;
\item $\pdim \Y_0 \leq 1$ if and only if $\pdim \Y \leq 1$.
\end{enumerate}
\end{proposition}

\begin{proof}
1. Let $X\in\X$ and $Y_0\in\Y_0$. Since $Y_0$ is finitely presented and $X=\ld X_\alpha$ for a directed family
of submodules $X_\alpha\in\X_0$ of $X$, we have
$$
\Ext^1_R(Y_0,X) = \Ext^1_R(Y_0,\ld X_\alpha) \cong \ld \Ext^1_R(Y_0,X_\alpha).
$$
Now, if $\Ext^1_R(\Y_0, \X_0)=0$, we derive that $\Ext^1_R(\Y_0, \X)=0$. Since $C$ is $\Sigma$-pure injective, by Lemma~\ref{lemma:limiti} we have $\mathcal Y^\perp=\mathcal Y_0^\perp$. Therefore $\Ext^1_R(\Y, \X)=0$.

2. Similarly, if $\Ext^2_R(\Y_0, \rMod R)=0$, by dimension shifting and using the fact that $\Y_0^\perp = \Y^\perp$
we see that $\Ext^2_R(\Y, \rMod R)=0$.
\end{proof}


As we mentioned in the previous section, in order to guarantee that a splitting torsion pair $(\mathcal X_0,\mathcal Y_0)$ in $\rmod R$ gives rise to a splitting torsion pair $(\mathcal X,\mathcal Y)=(\varinjlim X_0, \varinjlim Y_0)$, Reiten and Ringel in \cite{RR} introduced the  Condition~\ref{RRcond}.   
Here we prove  that, for a noetherian ring,  the Reiten-Ringel  condition  completely characterizes torsion pairs cogenerated by a $\Sigma$-pure injective 1-cotilting module.

\begin{lemma}\label{SubFil}
The class of all $\Y_0$-filtered modules is closed under submodules.
\end{lemma}
\begin{proof}
Suppose that $Y$  is $\Y_0$-filtered by $(Y_\lambda \mid \lambda\leq\mu)$. Then any $Y'\leq Y$ is
$\Y_0$-filtered by $(Y'_\lambda=Y_\lambda\cap Y' \mid \lambda\leq\mu)$. Indeed this is a continuous
chain starting from $0$ and ending to $Y'$ such that for every $\lambda<\mu$
$$
\dfrac{Y'_{\lambda+1}}{Y'_\lambda} = \dfrac{Y_{\lambda+1}\cap Y'}{Y_\lambda\cap Y_{\lambda+1} \cap Y'}
\cong \dfrac{Y_\lambda+(Y_{\lambda+1}\cap Y')}{Y_\lambda} \leq \dfrac{Y_{\lambda+1}}{Y_\lambda} \in\Y_0,
$$
so that $Y'_{\lambda+1}/Y'_{\lambda} \in\Y_0$.
\end{proof}

\begin{theorem}\label{thm:torsionpairs}
Let $R$ be a right noetherian ring. For a cotilting  torsion pair $(\mathcal X, \mathcal Y$) in $\rMod R$ the following conditions are equivalent:
\begin{enumerate}
\item $\mathcal Y$ satisfies the Reiten-Ringel condition;
\item any 1-cotilting module cogenerating $(\mathcal X, \mathcal Y)$
 is $\Sigma$-pure injective;
 \item $(\mathcal X, \mathcal Y)$ is cogenerated by a $\Sigma$-pure injective 1-cotilting module.
\end{enumerate}
\end{theorem}
\begin{proof}
$1\Rightarrow 2$ Follows from   Proposition~\ref{prop:Ringelcondition} and Corollary~\ref{cor:2.5}.

$2\Rightarrow 3$ is obvious.

$3\Rightarrow 1$  Let us assume that $C$ is $\Sigma$-pure injective. By Corollary~\ref{cor:2.5}, we have $\mathcal Y={}^\perp(\mathcal Y_0^\perp)$. By \cite[Corollary~3.2.4]{GT} and Lemma~\ref{SubFil}, any module in $\mathcal Y$ is $\mathcal Y_0$-filtered. Let $Y$ be a module in $\mathcal Y$ and $F$ a finitely generated submodule of $Y$ such that $Y/F$ is torsion. By Theorem~\ref{lemma:hill}, $F$ is contained in a finitely generated submodule $\overline F$ of $Y$ such that $Y/\overline F$ belongs to $\mathcal Y$. Since $Y/\overline F$ is a quotient of $Y/F$, it is also torsion and hence $Y=\overline F$ is finitely generated.
\end{proof}

The $\Sigma$-pure injectivity of $C$ is equivalent to some other interesting finiteness condition on $\Y$.

\begin{proposition}
Let $C$ be a 1-cotilting module and $(\mathcal X,\mathcal Y)$ the torsion pair cogenerated by $C$. Then the following conditions are equivalent:
\begin{enumerate}
\item $C$ is $\Sigma$-pure injective;
\item if $Y\in\Y$ has a finitely generated submodule $0\neq Y_0\leq Y$ such that $Y/Y_0\in\X$, then $Y$ is $\Y_0$-filtered;
\item there are no infinite strictly ascending chains $Y_0<Y_1<\dots<Y_i<Y_{i+1}< \dots$
in $\Y_0$ such that
$Y_{i+1}/Y_{i}\in \X_0$ (equiv.~$Y_{i+1}/Y_{0}\in \X_0$) for all $i\in\N$;
\item if $Y_0$ is a finitely generated submodule of $Y\in\mathcal Y$, then the torsion part of $Y/Y_0$ is finitely generated;
\item any non-zero module $Y\in\Y$ has a finitely generated non-zero submodule $Y^+\leq Y$ such
that $Y/Y^+\in\mathcal Y$.
\end{enumerate}
\end{proposition}
\begin{proof}
$1\Rightarrow 2$: it is clear, since any module in $\mathcal Y$ is $\mathcal Y_0$-filtered.

$2\Rightarrow 3$: assume that $(Y_i)_{i\in\N}$ is a infinite strictly ascending chains $Y_0<Y_1<\dots<Y_i<Y_{i+1}< \dots$
in $\Y_0$ such that
$Y_{i+1}/Y_{i}\in \X_0$  for all $i\in\N$. Then $Y=\cup Y_i=\ld Y_i\in\Y$ is not finitely generated and $Y/Y_0 = (\ld Y_i)/Y_0 \cong \ld (Y_i/Y_0) \in \X$. Therefore $Y$ is $\mathcal Y_0$-filtered and by Theorem~\ref{lemma:hill}, $Y_0$ is contained in a finitely generated submodule $\overline{Y_0}$ of $Y$ such that $Y/\overline{Y_0}$ is in $\mathcal Y$. Since $Y/\overline{Y_0}$ is a quotient of $Y/Y_0$, and hence it is also torsion, we  have $Y=\overline{Y_0}$, contradicting the fact that $I$ is not finitely generated.

$3\Rightarrow 4$: suppose that for a given $ Y$ in $\mathcal Y$, and a finitely generated submodule $F$ of $Y$, the quotient $Y/F$ has torsion part $\overline Y/F$ which is not finitely generated. Then $\overline Y/F$ is a direct limit of modules $X_i$, $i\in I$, in $\mathcal X_0$. Denoting by $Y_i/F$ the homomorphic images in $\overline Y/F$ of the modules $X_i$, by means of the $Y_i$'s one can construct an infinite strictly ascending chain in $\mathcal Y_0$, contradicting 3.

$4\Rightarrow 5$: given any finitely generated submodule $Y_0$ of $Y$, take as  $Y^+$ the submodule of $Y$ containing $Y_0$ such that $Y^+/Y_0$ is the torsion part of $Y/Y_0$.

$5\Rightarrow 1$: we will prove that any module in $\mathcal Y$  is $\mathcal Y_0$-filtered. By \cite[Corollary~3.2.4]{GT} and Lemma~\ref{SubFil}, we will get $\mathcal Y={}^\perp(\mathcal Y_0^\perp)$, so that $C$ is $\Sigma$-pure injective by Corollary~\ref{cor:2.5}.

Let $Y$ in $\mathcal Y$ and $\mu = 2^{|Y|}$.  We construct, by transfinite induction, a continuous chain $(Y_\lambda \mid \lambda\leq\mu)$
of submodules of $Y$ such that, for every $\lambda<\mu$,

i) $Y/Y_{\lambda}\in\Y$,

ii) $Y_{\lambda+1}/Y_\lambda\in\Y_0$,

iii) if $Y_\lambda \lneqq Y$, then $Y_\lambda \lneqq Y_{\lambda+1}$.
\par\noindent
Set $Y_0=0$, and set $Y_{\lambda+1}=Y_\lambda$ in case $Y_\lambda=Y$,
or $Y_{\lambda+1}/Y_\lambda = (Y/Y_\lambda)^+$ if $Y_\lambda \lneqq Y$
(if this is the case, then $Y/Y_{\lambda+1} \cong (Y/Y_\lambda)/(Y/Y_\lambda)^+\in\Y$).
Then conditions i), ii) and iii) are clearly satisfied. Finally, if $\lambda$ is a limit ordinal, we set
$Y_\lambda = \bigcup_{k<\lambda}Y_k$, and condition iii) holds since
$Y/Y_\lambda = \ld_{k<\lambda} Y/Y_k\in\Y$.
This defines a $\Y_0$-filtration $(Y_\lambda \mid \lambda\leq\mu)$.
Finally, the choice $\mu = 2^{|Y|}$ and the property iii) guarantee that $Y_\mu = Y$.
\end{proof}

\begin{remark}
If $\Lambda$ is a  tame hereditary $k$-algebra over an algebraically closed field $k$, \cite{BK} gives a complete description of the cotilting $\Lambda$-modules, up to equivalence.  The $\Sigma$-pure injective ones are exactly those with no adic direct summand.  Thus by Theorem~\ref{thm:torsionpairs}, the cotilting torsion pairs which satisfy the Reiten-Ringel Condition~\ref{RRcond} are completely determined.
\end{remark}




\begin{thebibliography}{99}
\bibitem{AST} Angeleri Huegel, L., $\check{\text{S}}$aroch, J and Trlifaj, J. \emph{On the telescope conjecture for module categories},  J. Pure Appl. Algebra  \textbf{212}  (2008), 297--310.
\bibitem{ATT} Angeleri Huegel, L., Tonolo, A. and Trlifaj, J. \emph{Tilting preenvelopes and cotilting precovers}, Algebr. Represent. Theory  \textbf{4}  (2001), 155--170.
\bibitem{B} Bazzoni, S. \emph{Cotilting modules are pure injective}, Proc. Amer. Math. Soc. \textbf{131} (2003), 3665--3672.
\bibitem{BK} Buan, A. and Krause, H. \emph{Cotilting modules over tame hereditary algebras}, Pacific J. Math. \textbf{211} (2003), 41--60.
\bibitem{C} Colpi, R. \emph{Tilting in Grothendieck categories}, Forum Math. \textbf{11} (1999), 735--759.
\bibitem{CF} Colpi, R., Fuller K. \emph{Tilting objects in abelian categories and quasitilted rings}, Trans. Amer. Math. Soc. \textbf{359} (2007), 741--765.
\bibitem{CGM} Colpi~R., Gregorio~E. and Mantese~F. \emph{On the heart of a faithful torsion theory}, J. Algebra \textbf{307} (2007), 841--863.
\bibitem{CG} Colpi~R. and Gregorio~E. \emph{On the heart of a cotilting torsion pair}, preprint.
\bibitem{GR} Garcia Rozas, J.R. \emph{Covers and Envelopes in the Category of Complexes of Modules}, Chapman \& Hall/CRC, London (1999).
\bibitem{GT} G\"obel~R. and Trlifaj~J. \emph{Approximations and Endomorphism Algebras of Modules}, De Gruyter Expositions in Math. \textbf{41}, Berlin - New York 2006.
\bibitem{HRS} Happel~D., Reiten~I., Smal$\o$~S.O. \emph{Tilting in Abelian Categories and Quasitilted Algebras}, Memoirs of the A.M.S., vol. \textbf{575}, 1996.
\bibitem{JL} Jensen, C.U. and Lenzing, H. \emph{Model Theoretic Algebra}, Algebra, Logic and Applications \textbf{2}, Gordon \& Breach, Amsterdam 1989.
\bibitem{RR} Reiten, I. and Ringel, C.M. \emph{Infinite dimensional representations of canonical algebras}, Canad. J. Math.  \textbf{58} (2006), 180--224.
\bibitem{St} Stenstr\"om, B. \emph{Rings of Quotients}, Grund. math. Wiss. \textbf{217}, Springer, New York (1975).
\bibitem{S} $\check{\text{S}}$t'ov\'\i$\check{\text{c}}$ek, J. \emph{All $n$-cotilting modules are pure injective}, Proc. Amer. Math. Soc. \textbf{134} (2006), 1891 --1897.
\end{thebibliography}
\end{document}